\documentclass[11pt]{article}
\usepackage{amsfonts}
\usepackage{amssymb}
\usepackage{amsmath}
\usepackage{amsthm}

\usepackage{epsfig,epic,eepic}
\newcommand{\R}{\mathbb R}
\newcommand{\Q}{\mathbb Q}
\newcommand{\T}{\mathbf T}
\newcommand{\Z}{\mathbb Z}
\newcommand{\C}{\mathbb C}

\newcommand{\hop}{\vskip .2cm\noindent}
\newcommand{\hip}{\vskip .1cm\noindent}

\newtheorem{enonce}{}[section]

\newtheorem{theo}[enonce]{Theorem}
\newtheorem{cor}[enonce]{Corollary}
\newtheorem{prop}[enonce]{Proposition}
\newtheorem{lem}[enonce]{Lemma}

\theoremstyle{definition}
\newtheorem{defi}[enonce]{Definition}

\theoremstyle{remark}
\newtheorem{rema}[enonce]{Remark}

\title{Projective properties of Lorentzian surfaces}
\author{Pierre Mounoud}
\date{}
\begin{document}
 \maketitle
\begin{abstract}
We investigate projective properties of Lorentzian surfaces. In particular, we prove that if $\T$ is a non flat torus, then the index of its isometry group in its projective group is at most two. We also prove that  any topologically finite noncompact surface can be endowed with a metric having a non isometric projective transformation of infinite order.
\end{abstract}
\section{Introduction}
Two (pseudo-)Riemannian metrics are called projectively equivalent if they have the same unparametrized geodesics. Projectively equivalent metrics are related to first integrals of  geodesic flows. The link is given by the following theorem, already known from  Darboux 
 \cite[\S 608]{Darboux} and rediscovered  by Matveev and Topalov \cite{MatetTopa} (in Russian, see for example \cite{Bryant_Matveev} for a recent reference in English). 
\begin{theo}[Darboux]\label{theo_classique}
Two metrics (possibly of different signature) $g$ and  $\bar g$ on a surface $M$ are projectively equivalent
if and only if the function $I: TM \rightarrow \R$ defined by
$$I(v)= \left(\frac{\det(g)}{\det(\bar g)}\right)^{\frac{2}{3}} \bar g(v,v)$$
is an integral of the geodesic flow of $g$.
\end{theo}
Let us remark that the integral $I$ given by Theorem \ref{theo_classique} is quadratic (along fibers). Reciprocally, any non degenerate quadratic integral provides projectively equivalent metrics. It follows therefore from Clairaut Lemma, see Lemma \ref{lem_clairaut}, that if $g$ has a Killing field then the set of metrics projectively equivalent to $g$ is at least two dimensional. We start this paper by giving direct, but in our opinion interesting, consequences of this fact. For example, we increase the family of Lorentzian tori without conjugate points given in \cite{BetM1}, see Theorem \ref{theo_conj}. We also study the projective extensions of surfaces all of whose spacelike geodesic are closed that have a Killing field, see section \ref{sec_zoll}.

The rest of the article, is devoted to non affine projective transformations of Lorentzian surfaces.
We recall that a diffeomorphism of a pseudo-Riemannian manifold is called projective if it sends unparametrized geodesics on unparametrized geodesics. The projective transformations of a pseudo-Riemannian manifold $(M,g)$ form a Lie group 
 that contains the group of affine transformations 
(given by the diffeomorphisms preserving the connection i.e.\ the parametrized geodesics) that contains the group of isometries. 
Note that finding non affine projective transformations is the same as finding 
 projectively equivalent Lorentzian metrics with distinct Levi-Civita connections that are isometric.

The projective groups of compact Riemannian manifolds are now well understood, thanks notably to works from  Matveev \cite{Matveev_jdg} and Zeghib \cite{Zeghib}. The conclusion of these studies is a theorem {\it \`a la} Ferrand-Obata (compare with \cite{ferrand}):
\begin{theo}[\cite{Zeghib}]\label{theo_ghani}
Let $(M, g)$ be a compact Riemannian manifold. If $M$ is not a Riemannian
finite quotient of a standard sphere, 
then the projective group of $g$ 
is a finite extension of its affine group. 
\end{theo}
It is natural to wonder  if such a result still hold in the pseudo-Riemannian setting. In the Lorentzian case, an important step has been done in this direction by Bolsinov, Matveev and Rosemann: 
\begin{theo}[{\cite[Theorem 1.2]{Matveev_new}}]\label{BMR}
Let $(M, g)$ be the compact Lorentzian manifold. Then any projective vector field
on M is an affine vector field.
\end{theo}

Using the classification of Lorentzian tori whose geodesic flow admits a quadratic integral given by Matveev in \cite {Matveev_au_japon}, we prove in section \ref{subsec_transfocompact} the following extension of Theorem \ref{BMR} for surfaces:
\begin{theo}\label{theo_fini}
 Let $\T$ be a non flat Lorentzian torus. The index of the isometry group of $\T$ into the projective group of $\T$ is lower or equal to $2$.
\end{theo} 

In section \ref{sec_transfononcompact}, we consider, the case of non compact surfaces. We show that in this case the situation is completely different: non affine projective transformations of infinite order often exist.
\begin{theo}\label{theo_existe}
 Every orientable, non compact surface of finite topological type can be endowed with a Lorentzian metric having an affine group of infinite index in its projective group. Moreover, any metric projectively equivalent to it has the same property.
\end{theo}
Projective properties of non compact Lorentzian surfaces have already been studied by several authors. We just mention here \cite{Bryant_Matveev} by Bryant, Manno and Matveev and \cite{Matveev_mathann} (and its appendix) by Matveev. In these papers, the point of view is local (whereas the point of view of Theorem \ref{theo_existe} is clearly a global one). The authors provide local normal form of metrics admitting projective vector fields. 
Hence, they work on the disc and the vector fields obtained are not always complete and therefore do not always generate projective transformations.

I wish to thank C.\ Boubel for making me realize that my first ideas were heading nowhere.

\section{Projectively equivalent metrics given by Killing fields}\label{section_examples}
\subsection{Clairaut Lemma}
Let us recall first the very classical Clairaut Lemma:
\begin{lem}[Clairaut Lemma]\label{lem_clairaut}
 Let $(M,g)$ be a pseudo-Riemannian manifold and $K$ be a Killing field of $(M,g)$. 
 The function $C:TM\rightarrow \R$ defined by $C(v)=g(K,v)$ is a (linear) integral of the geodesic flow.
\end{lem}
Keeping the notations of Lemma \ref{lem_clairaut}, we note that if $C$ is bounded then  for small $t\in \R$, the form $J=g+tC^2$ is a non degenerate quadratic integral of the geodesic flow of $g$. It follows from Theorem \ref{theo_classique} that the metric $\bar g$ defined by:
\begin{equation}\label{eq_formule}\bar g(v,v)=\left(\frac{\det(g)}{\det(J)}\right)^{\frac{2}{n-1}} J(v)\end{equation}
is projectively equivalent to $g$.

If $\dim M=2$, on the neighborhood on any point that is not a saddle point of $K$ there exists coordinates $(x,y)$ and a function $f$ such that the metric reads $2dxdy+f(x)dy^2$ (see \cite{BetM} for an account on surfaces with a Killing field). The metric projectively equivalent to $g$ obtained thanks to Clairaut integral are easily given. 
\begin{prop}\label{prop_proj_bande}
  Let $I\subset \R$ be an interval and $f:I \rightarrow \R$ be a smooth function. For any $(a,\ell)\in \R^*\times \R$ such that $1+\ell f$ does not vanish on $I$, the metrics $g$ and $g_{a,\ell}$ defined on $I\times \R$ respectively by $2dxdy+f(x)dy^2$ and $\frac{a\ell}{(1+\ell f(x))^2} dx^2 +\frac{2a}{1+\ell f(x)}dxdy +\frac{af(x)}{1+\ell f(x)}dy^2$ have the same unparametrized geodesics.
 \end{prop}
\begin{proof} 
Let $a,\ell,z$ be numbers such that $1+\ell\,z\neq 0$, and
 \begin{eqnarray*} G_{a,\ell}(z)=a\begin{pmatrix}\frac{\ell}{(1+\ell\,z)^2} & \frac{1}{1+\ell\,z}\\ \frac{1}{1+\ell\,z} & \frac{z}{1+\ell\,z}\end{pmatrix}, \\
 Q_{a,\ell}(z)=\frac{a^2}{(1+\ell\,z)^2}\begin{pmatrix} 1 &z\\z&z^2\end{pmatrix}.
 \end{eqnarray*}
The proof follows directly from Theorem \ref{theo_classique} and   Lemma \ref{lem_matrice_equiv} applied to $G_{1,0}(f(x))$. \end{proof}
  \begin{lem}\label{lem_matrice_equiv}
 Setting $b=a\beta^{-3}$ and $m=\ell+\mu a$, we have:
 \begin{equation}\label{eq_matrice_equiv} \beta( G_{a,\ell }(z) +\mu \,Q_{a,\ell }(z))= \left(\frac{\det (G_{a,\ell }(z))}{\det(G_{b,m}(z))}\right)^{\frac{2}{3}}G_{b,m}(z).\end{equation} 
 \end{lem}
 \begin{proof}
 It is a direct computation, on one hand we have:
$$ \beta( G_{a,\ell }(z) +\mu \,Q_{a,\ell }(z))=
\frac{\beta a}{(1+\ell\,z)^2}
\begin{pmatrix}
\ell+\mu a & 1+ z(\ell+\mu a)\\
1+z(\ell+\mu a) &  z(1+ z(\ell+\mu a)) 
\end{pmatrix}
$$
and on the other hand, as $\det(G_{a,\ell }(z))=\frac{-a^2}{(1+\ell z)^3}$, we have:

$$\left(\frac{\det (G_{a,\ell }(z))}{\det(G_{b,m}(z))}\right)^{\frac{2}{3}}G_{b,m}(z)=\left(\frac{a}{b}\right)^{\frac{4}{3}} \frac{(1+m z)^2}{(1+lz)^2} \frac{b}{(1+m z)^2} \begin{pmatrix}
m & 1+mz\\    (1+mz) &    z(1+mz)                                                                                                                                                                   
 \end{pmatrix}.
$$
When $b=a\beta^{-3}$ and $m=\ell+\mu a$, these two matrices are clearly equal.
\end{proof}
\begin{rema}\label{rem_perp}
  The distribution orthogonal for $g_{a,\ell}$ to the Killing field does not depend on $a$ and $\ell$. It means in particular that these metrics have the same ``generic reflections''. A generic reflection of a surface with a Killing $K$ is a local isometry fixing a non degenerate geodesic perpendicular to $K$ and sending $K$ to $-K$ (as a consequence it  permutes the lightlike foliations). In general it is only defined on the saturation of the geodesic by the flow of $K$, see \cite{BetM} for details. 
  
  The metric $g_{a,\ell}$ has a lightlike geodesic cutting each lightlike orbit of $K$ (what is called a ``ribbon'' in \cite{BetM}), it implies that there exits coordinates $(u,v)$ and a function $f_{a\ell}$ such that it reads $2dudv + f_{a,\ell}(u) dv^2$. According to \cite{BetM}, if $K$ has a lightlike orbit, each of these metrics extends into a bigger surface denoted $E^u_{f_{a,\ell}}$ which is only determined by the function $f_{a,\ell}$. These surfaces are still projectively equivalent (some could even be isometric). We can note that it defines an equivalence  relation on the set of real functions. 
 \end{rema}
\subsection{Surfaces without conjugate points}
Theorem \ref{theo_classique} allows us to extend the family of metrics on $\R^2$ inducing tori without conjugate points. 
\begin{theo}\label{theo_conj}
 There exists a $2$-dimensional family of Lorentzian tori without conjugate points with pairwise non isometric universal cover.
\end{theo}
\begin{proof}
 The main result of \cite{BetM1} asserts that the Clifton-Pohl plane: $(\R^2\smallsetminus\{0\},\frac{2}{x^2+y^2}dxdy)$ has no conjugate points. If follows from the discussion above that there exists a two dimensional space of metrics projectively equivalent metrics the Clifton-Pohl plane.
Corollary 5.6 of \cite{Piccione} implies that  having conjugate points is a projective property of  Lorentzian surfaces, therefore none of these metrics has conjugate points. The radial field of $\R^2$ is a Killing field of each of these metrics, therefore they all have a quotient diffeomorphic to the torus.
  
 From \eqref{eq_formule},  these metrics reads on $\R^2\smallsetminus \{0\}$:
 $$
 g_{a,\ell}=\left(\frac{1}{2abxy+a^2(x^2+y^2)}\right)^2 by^2dx^2 +2(bxy+a(x^2+y^2))dxdy + bx^2dy^2,
 $$
 where $a,\ell$ are real numbers satisfying $-a<\ell<a$. 
 We prove now that  $g_{a,\ell}$ is isometric to $g_{a',\ell'}$ if and only if $|a|=|a'|$ and $\ell=\ell'$. 
 As the map $(x,y)\mapsto (-y,x)$ sends $g_{a,\ell}$ on $g_{-a,\ell}$, we can assume now that $a$ and $a'$ are both positive.
  For any $a,\ell$, we have $g_{a,\ell}(K,K)=0$ if and only if $xy=0$.  Computing the curvature $\rho_{a,\ell}$ at these points we find $-a^2\ell$. As $a>0$, the norm of $K$ is maximal when $x=y$ and $\rho_{a,\ell}(x,x)=-2a^2(a+\ell)$.
 If there exits an isometry $\Phi$ between $g_{a,\ell}$ and $g_{a',\ell'}$, it preserves the set of lightlike orbits of $K$ and the set where $K$ has maximal norm. Consequently we have:
 $$
 \left\{
 \begin{array}{l}
 a^2\ell=a'^2\ell'\\
 a^2(a+\ell)=a'^2(a'+\ell')
 \end{array}
 \right.
 $$
 and therefore $(a,\ell)=(a',\ell')$.
\end{proof}
\begin{rema} 
Each metric $g_{a,\ell}$ admits a non isometric projective transformation (which is affine if $\ell=0$): the involution $(x,y)\mapsto (-y,x)$. It will follow from Theorem \ref{theo_fini} that the index of the isometry group of the Lorentzian torus associated to $g_{a,\ell}$  (which is equal to its affine group when $\ell\neq 0$) in its projective group is $2$.
\end{rema}
\subsection{Surfaces all of whose spacelike geodesics are closed}\label{sec_zoll}
Another interesting family of Lorentzian surfaces is given by the set of surfaces all of whose spacelike geodesics are closed. Some of them have also a Killing field, see \cite{MounetSuhr_jdg}, therefore there exists metrics projectively equivalent to them. However, if two surfaces all of  whose spacelike geodesics are closed are projectively equivalent then they have the same lightlike geodesics. Indeed, the lightlike geodesics are the first non closed geodesics of the surface, see \cite{MounetSuhr_MZ}. It means they are in the same conformal class, but two conformal metrics that are projectively equivalent are proportional, see \cite[Theorem 1.2.3]{Taber} for example. Hence, the metrics obtained always have a non closed spacelike geodesic. We  easily deduce from this fact that a projective transformation of a Lorentzian surface all of whose spacelike geodesics are closed is an isometry.

Let us see what happens in the constant curvature case ie on the de Sitter surface. Let $q$ be quadratic form on $\R^3$ of signature $(2,1)$. We recall that de Sitter surface is (isometric to) the quadric $S_q=q^{-1}(1)$ endowed with the metric induced by $q$. Its geodesics are thus the intersections of the vectorial planes with $S_q$. Let  $q_1$ and $q_2$ be two non degenerate quadratic forms, the signature of $q_1$ being $(2,1)$. The radial field of $\R^3$ defines a projection from a subset of $S_{q_1}$ to $S_{q_2}$. In general it is not everywhere defined, but it always sends geodesics of $S_{q_1}$ on geodesics of $S_{q_2}$. 
When the lightcone of $q_2$ is contained in the lightcone of $q_1$ then it is everywhere defined but not surjective. Otherwise said $S_{q_1}$ is projectively equivalent to a proper open subset of $S_{q_2}$. 
Hence, the de Sitter surface is projectively equivalent to surfaces that 
do not have all their spacelike geodesics closed but have an extension having this property. Note that we did not assume $q_2$ to be of Lorentzian type, it could be positive definite.
We wonder if any surface projectively equivalent to (an open subset of) a metric all of whose spacelike geodesic are closed can be extended into a metric all of whose  spacelike (or timelike) geodesics are closed. 

Following \cite{Besse}, we give the following definition (we warn the reader that the expression Tannery surface is sometimes used with a slightly different meaning, without reference to a Killing field).
\begin{defi}\label{def_tan}
 An pseudo-Riemannian surface is said to be a Tannery surface, if it is orientable, it has a periodic spacelike \footnote{in the Riemannian case, we call  any non zero vector spacelike} Killing field, and  if all its spacelike geodesics are closed except possibly the one orthogonal to the Killing field (in the Riemannian case).
 \end{defi}
Lorentzian Tannery surfaces correspond to the surfaces of elliptic type of \cite{MounetSuhr_jdg}.
Next theorem says in particular that Tannery surfaces have a maximal projective (Tannery) extension.
\begin{theo}\label{theo_zoll}
 Any Lorentzian Tannery surface is projectively equivalent to an open subset of a Riemannian Tannery surface. Reciprocally any Riemannian Tannery surface contains a subset projectively equivalent to a Lorentzian Tannery surface.
\end{theo}
\begin{proof}
 According to \cite[Theorem 4.13]{Besse}, if $g_0$ is a Tannery surface then there exists global  coordinates $(r,\theta) \in ]0,\pi[\times S^1$ on the surface, $\frac pq\in \Q$  and an odd function $h$ such that $g_0$ reads:
 $$g_0=\left(\frac{p}{q}+h(\cos r)\right)^2dr^2+\sin^2r\, d\theta^2.$$
 It follows from Theorem \ref{theo_classique}, that for any $\ell\in \R$ the metric
 $$\frac{-1}{(1+\ell \sin^2 r)^2}\left(\left(\frac{p}{q}+h(\cos r)\right)^2dr^2+\sin^2r(1+\ell \sin^2 r)\, d\theta^2 \right) $$
 is projectively equivalent to $g_0$ on its domain of definition. If we restrict $r$ to $]\pi/4,3\pi/4[$ then $\sin^2 r\in ]1/2,1[$, we can therefore take $\ell=-2$. The metric $g_1$  obtained is Lorentzian.
 A unitary geodesic $\gamma$ of $g_0$ is lightlike for $g_1$ if and only if 
 $$1-2\sin^4 (r) \dot \theta^2=0.$$
 Moreover, $\gamma$ is tangent to the parallels $\{r=\pi/4\}$ and $\{r=3\pi/4\}$ if and only if $g_0(\dot\gamma,\partial_\theta)= \pm \sin(\pi/4)$. But $g_0(\dot\gamma,\partial_\theta)=\sin^2(r) \dot\theta$ so these conditions are the same. It means that all the spacelike geodesics of $g_1$ are closed.
\hop 

Let us see now that the metric $g_1$ can be any Lorentzian Tannery surface. Let $x:\R\rightarrow \R$ defined by 
$$
x(t)=\left\{\begin{aligned}
&             \arcsin\left(\sqrt{\frac{-\cosh(t)^2}{1-2\cosh(t)^2}}\right),&\text{if}\ t\leq 0\\
&             \pi- \arcsin\left(\sqrt{\frac{-\cosh(t)^2}{1-2\cosh(t)^2}}\right),&\text{if}\ t\geq 0.
            \end{aligned}
%
\right.
$$
 It is smooth and made to satisfy $\frac{\sin^2(a(t))}{1-2\sin^2(a(t))}=-\cosh^2(t)$.
In the coordinates $(t,\theta)$, the metric reads: 
$$-\left(\frac{p}{q}+(h\circ v) (\sinh t)\right)^2dt^2 + \cosh^2(t)\, d\theta^2,$$
where $v(s)=\frac{s}{\sqrt{1+2s^2}}$. The function $h\circ v$ being also odd and $v$ being a diffeomorphism,
by a trivial, but rather long, adaptation of chapter 4 of \cite{Besse} (unfortunately it is not the description used in \cite{MounetSuhr_jdg}), we see that any Lorentzian Tannery  surface can be written this way.
\end{proof}
One half of Theorem \ref{theo_zoll} is easily generalizable to Lorentzian surfaces with an other type of Killing field.
\begin{prop}\label{prop_contains}
 Let $(M,g)$ be an orientable Lorentzian surface with a Killing $K$. If the  spacelike geodesics  of $g$ are all closed, then $(M,g)$ contains a proper open subset projectively equivalent to a Lorentzian surface all of whose spacelike geodesics are closed. 
\end{prop}
\begin{proof}
 It is proven in \cite[Propositions 3.2 and 3.4]{MounetSuhr_jdg} that $\sup g(K,K)=+\infty$ and that any spacelike geodesic is somewhere tangent to $K$. Let $\ell$ be a positive number and $U=\{p\in M\,|\,  g(K,K)<\ell^2\}$, it follows from \cite[Proposition 3.4]{MounetSuhr_jdg} that $U$ is connected. 
 The boundary of $U$ is given by a union of orbits of $K$. A unitary geodesic $\gamma$ is tangent to $\partial U$ if and only if $C(\dot \gamma)=\ell$,  where $C$ is the Clairaut integral associated to $K$.
  From Proposition \ref{prop_proj_bande}, we deduce that  $J=g-\frac{1}{\ell^2}C^2$ is non degenerate on $U$.  We endow $U$ with the metric $$\bar g=\left(\frac{\det(g)}{\det(J)}\right)^{\frac{2}{3}}J,$$ it is  projectively equivalent to $g$.  A vector $v$ such that $g(v,v)=1$ satisfies $J(v)=0$ if and only if $C(v)=\ell$. It implies as above that the spacelike geodesics of $(U,\bar g)$ are all closed. 
\end{proof}
\begin{rema}
Unfortunately, Proposition \ref{prop_contains} cannot be used to deform one of the spacelike Zoll surfaces given in \cite{MounetSuhr_jdg} into a spacelike Zoll surface of a different kind. Indeed, it is not difficult to see that the families of surfaces constructed there are precisely the surfaces having the same Clairaut integral as a metric of positive constant curvature. Therefore the metrics obtained have the same Clairaut integral as a metric obtained by deformation of a metric of positive constant curvature. But these metrics also have positive constant curvature and therefore the metrics obtained by deformation are also of the kind described in \cite{MounetSuhr_jdg}.
\end{rema}
\section{Projective transformations}
\subsection{Non compact surfaces, proof of Theorem \ref{theo_existe}}\label{sec_transfononcompact}

 We start by proving the theorem on $\R^2$.
 Let $f$ be $1$-periodic non negative and non constant  function and let $m$ be its maximum. Let $a$ be a real number  greater than $1$, let $\varepsilon$ be a positive number satisfying $\varepsilon< (a^3-1)/ma^3 $.
 
 Let $\alpha: [0,1]\rightarrow [1,a]$ be a smooth function equal to $1$ on a neighborhood of $0$ and equal to $a$ on a neighborhood of $1$. 
 Let $\lambda: [0,1]\rightarrow [0,1]$ be a smooth function equal to $0$ on a neighborhood of $0$ and equal to $\varepsilon$ on a neighborhood of $1$.

 For any $n\in \Z$, we define two functions $A_n$ and $\Lambda_n$ on $[n,n+1]$ by 
 $$A_n(x+n)=\alpha(x)a^n$$
 $$\Lambda_n(x+n)= \lambda(x)+\varepsilon\,\frac{1-a^{3n}}{1-a^3}\alpha^3(x)$$
 Let $A$ (respectively  $\Lambda$) be the function such that for any $n\in \Z$ the restriction of $A$ (resp. $\Lambda$) to $[n,n+1]$ is equal to $A_n$ (resp. $\Lambda_n$). The functions  $A$ and $\Lambda$ are smooth. We see also that for any $x\in \R$, $\Lambda(x)>-\frac{1}{m}$ and therefore $1+ \Lambda(x)f(x)>0$. Hence the tensor defined by
 $$g = A^3(x)\left(\frac{\Lambda(x)}{(1+ \Lambda(x)f(x))^2} dx^2 +\frac{2}{1+\Lambda(x) f(x)}dxdy +\frac{f(x)}{1+\Lambda(x)f(x)}dy^2\right)$$
 is a Lorentzian metric on $\R^2$.
 
When $\ell=\Lambda(x)$, $a=A^3(x)$ and $z=f(x)$, the matrix $G_{a,\ell }(z)$ defined in section \ref{section_examples} is equal to $G_x$ the matrix of  $g$ at any point $(x,y)$ and $Q_{a,\ell }(z)$ is the matrix of the square of Clairaut integral associated to the Killing field $\partial_y$.

Let $x$ be a point of $[0,1]$. Notice that for any $n\in \Z$, $A(x+n)=a^{3n}A(x)$ and $\Lambda(x+n)=\Lambda(x)+\varepsilon \frac{1-a^{3n}}{1-a^3}A^3(x)$. Therefore relation  \eqref{eq_matrice_equiv} reads  
$$a^{-n}(G_x+\varepsilon \frac{1-a^{3n}}{1-a^3} C^2(x))=\left(\frac{\det (G_x)}{\det(G_{x+n})}\right)^{\frac{2}{3}}G_{x+n},$$
where $C$ is the Clairaut integral.
Thus Theorem \ref{theo_classique} 
implies that the map $\tau: (x,y)\mapsto (x+1,y)$ is a projective map. It is clearly not an isometry. 
As $g$ is not flat the only parallel endomorphisms of the tangent bundle are the multiples of the identity (see Fact page 2226 of \cite{BoubetMou} for a proof). 
It means that any metric having the same Levi-Civita connection as $g$ is proportional to $g$.
However, the vector field $\partial_x$ is lightlike when $x=0$ but not when $x=1$, so $\tau^*g$ is not proportional to $g$. Consequently $\tau$ is not an affine map.
\hip
If $\bar g$ is a metric projectively equivalent to $g$ then Theorem \ref{theo_classique} says that it defines an integral $I$ of the geodesic flow of $g$. If $I$ were not a linear combination of the energy and Clairaut integral, then by a straightforward adaptation of Lemma 2 of \cite{Bryant_Matveev}, the dimension of the space of projective vector fields of $g$ would be at least $2$. But Theorem 1 of \cite{Bryant_Matveev} gives a local normal form for such metrics. It is clearly possible to find a function $f$ such that the metric $g$ is not isometric to one of these metrics (in fact the opposite seems to be impossible).  Consequently, $\tau$ does not preserve $\bar g$.
\hop

We choose now $f$ such that $f^{-1}(0)=\Z$. We choose also a maximal geodesic $\gamma_0$ in $]0,1[\times \R $ perpendicular to $\partial_y$. For any $n\in \Z$, we denote by $\gamma_n$ the image of $\gamma_0$ by $\tau^n$. Each curve $\gamma_n$ is perpendicular to $\partial_y$, see Remark \ref{rem_perp}. Let $\sigma_n$ be the generic reflection associated to $\gamma_n$, ie the isometry of $]n,n+1[\times \R$ sending $\partial_y$ to $-\partial_y$ that fixes $\gamma_n$. We clearly have 
\begin{equation}\label{eq_commute}
\tau\circ\sigma_n=\sigma_{n+1}\circ \tau.
\end{equation}
We consider then the Lorentzian manifold $X$ obtained by gluing two copies of $(\R^2,g)$ along each set $]n,n+1[\times \R$ thanks to $\sigma_n$ (the fact that $X$ is a Hausdorff Lorentzian surface is proven in \cite[Remarque 3.4]{BetM}).  It follows from \eqref{eq_commute} that $\tau$ induces a well defined projective transformation of $X$. Finally, Proposition 4.31 of \cite{BetM} tells us that any non compact, orientable surface of finite topological type can be realized as a quotient of the universal cover of $X$. Equation \eqref{eq_commute} tells us that the lift of $\tau$ belongs to the normalizer of the group of isometries used to construct these surfaces (the so-called "generic subgroup" in \cite{BetM}), therefore $\tau$ induces a projective transformation on them.

\begin{rema}
We denote by  $\Gamma$  the group generated by $\tau$ and the map $(x,y)\mapsto (x,y+1)$. This group preserves the projective connection of $(\R^2,g)$ (the metric defined in the proof of Theorem \ref{theo_existe}). Hence the torus $\R^2/\Gamma$ is a torus endowed with a projective connection that is locally but not globally Lorentzian.

The function $g(\partial_y,\partial_y)$ tends always to $0$ when $x$ goes to $-\infty$, therefore, at least  when the geodesic flow of $g$ does not admit another first integral (what is probably always the case), there is no Riemannian metric projectively equivalent to $g$.
\end{rema}
\subsection{Compact surfaces, proof of Theorem \ref{theo_fini}}\label{subsec_transfocompact}
The starting point of the proof of Theorem \ref{theo_fini} is the fact proven by 
 Matveev  see \cite[Theorem 2]{Matveev_au_japon} that a compact Lorentzian surface whose geodesic flow admits a quadratic integral  either admits a Killing field or is of Liouville type (see below for a definition). 

Let us assume that there exits a non flat Lorentzian torus $\T$ whose isometry group has an index greater than two in its projective group. Let us assume first that it has also a non trivial Killing field $K$.
\begin{lem}\label{lem_KsurK}
 Let $g$ be a non flat Lorentzian metric on $\T$ that possess a non trivial Killing field $K$. If the index of  ${\rm Isom}(\T)$ in ${\rm Proj} (\T)$ is greater than 2, then  
$g(K,K)$ is non negative or non positive  and  vanishes and  $\Phi_*K=\pm K$ for any $\Phi\in {\rm Proj}(\T)$. 
\end{lem}
\begin{proof} Let $\Phi$ be a projective transformation of $\T$.
 As $g$ is not flat, it follows from \cite[Theorem 6]{Matveev_au_japon} that the space of projective vector fields of $g$ is spanned by $K$. As $\Phi_*K$ is a projective vector field (its flow is a composition of projective transformations), there exists $c\in \R^*$ such that $\Phi_*K=cK$. The flow of $K$ being periodic (see for example \cite{BetM}), $|c|$ has to be $1$. 
 
 The metric $\Phi^*g$ is projectively equivalent to $g$ therefore, by Theorem \ref{theo_classique}, 
 \begin{equation}\label{eq_I*}I= \left(\frac{\det(g)}{\det(\Phi^* g)}\right)^{2/3} \Phi^*g\end{equation}
 is an integral of the geodesic flow of $g$. As $g$ is not flat and has a non trivial Killing field, it follows from \cite[Corollary 1]{Matveev_au_japon} that $I$ is a linear combination of $g$ and the square of Clairaut integral. Hence denoting by $C$ the Clairaut integral there exist $a\neq 0$ and $\ell$ such that  $I=a^{-1/3}( g +\ell C^2)$. As $\Phi$ preserves $\pm K$, the minimum (resp.  maximum) of the functions $g(K,K)$ and $\Phi^*g(K,K)$ are equal. We denote them by $m_-$ and $m_+$. 
 

We denote by $\text{vol}_g$ the volume form of $g$. The $1$-form $i_K\text{vol}_g$ is closed and induces a nowhere zero closed $1$-form on the space of leaves of $K$ (that is diffeomorphic to the circle). We choose a coordinate $x$ on this space such that $dx=i_K\text{vol}_g$. We call such a coordinate a canonical coordinate associated to $(g,K)$.
Note that the coordinate $x$ that appears in Proposition \ref{prop_proj_bande} induces a canonical coordinate associated to $(g,K)$ on the space of leaves.
Let $f:S^1\rightarrow \R$ be the expression of the function $g(K,K)$ in this coordinate. It follows from Proposition \ref{prop_proj_bande} that the expression of $\Phi^*g(K,K)$ in this coordinate is $\frac{af}{1+\ell f}$, where $a$ and $\ell$ are the numbers given above (note that  $x$ is not a canonical coordinate associated to $(\Phi^*g,K)$). It is therefore clear that the set of lightlike orbits of $K$ is preserved by $\Phi$ and that 
 $$\begin{array}{ll}
\text{if}\ a>0, &  \Phi(f^{-1}(m_+))=f^{-1}(m_+)\ \text{and}\ \Phi( f^{-1}(m_-))=f^{-1}(m_-)\\
\text{if}\ a<0, &  \Phi(f^{-1}(m_+))=f^{-1}(m_-)\ \text{and}\ \Phi( f^{-1}(m_-))=f^{-1}(m_+).
  \end{array}
$$
If the index of the isometry group in the projective group is greater than $2$, we can choose $\Phi$ such that leaves invariant both 
$f^{-1}(m_-)$ and  $f^{-1}(m_+)$. 
It implies that 
 $m_\pm=\frac{a m_\pm}{1+\ell m_\pm}$, and as $m_-\neq m_+$ (because the metric is not flat),  $m_+m_-=0$ and $a=1+\ell m_+$ or $1+\ell m_-$. 
\end{proof}

Let $\Phi$ be a projective transformation of $\T$ such that $\Phi^*g\neq \pm g$. 
Let $p$ be a point of $T$ belonging to a lightlike orbit of $K$ and let $q$ be the image of $p$ by $\Phi$. There exist a real function $f$, an interval $I$ (resp. $I'$) and  a neighborhood of $p$ (resp of $q$) that is isometric to $I\times S^1$ (resp. $I'\times S^1$) endowed with the metric $2dxdy +f(x)dy^2$. By Lemma \ref{lem_KsurK}, we can assume that $f$ is non negative.

According to Proposition \ref{prop_proj_bande} and Lemma \ref{lem_KsurK}, there exists $\ell\in \R$ such that  $\Phi^*g$ reads on each of these neighborhoods
$$(1+\ell m_+)\big(\frac{\ell}{(1+\ell f(x))^2} dx^2 +\frac{2}{1+\ell f(x)}dxdy +\frac{f(x)}{1+\ell f(x)}dy^2\big),$$
note that  $\ell\neq 0$ as $\Phi^*g\neq \pm g$.

Let $J$ be the Jacobian of $\Phi$ at  $p$ in the coordinates $(x,y)$. As the orbit of $K$ through $q$ is also lightlike, it satisfies:
$${}^tJ\begin{pmatrix} 0&1\\1&0\end{pmatrix}J=(1+\ell m_+) \begin{pmatrix} \ell &1\\1&0\end{pmatrix}$$
therefore, knowing that $J\begin{pmatrix} 0\\1\end{pmatrix}=\pm\begin{pmatrix} 0\\1\end{pmatrix}$,
$J=\pm \begin{pmatrix}
1+\ell m_+ &0\\
\ell/2 & 1                                                                                                                                                                   \end{pmatrix}$. 
The matrix $J^{-1}$ being the Jacobian of $\Phi^{-1}$ at $q$, the matrix of $\Phi^{-1}{}^*g$ at $q$ is given by:
$${}^tJ^{-1}\begin{pmatrix} 0&1\\1&0\end{pmatrix}J^{-1}= \frac{1}{1+\ell m_+} \begin{pmatrix} \frac{-\ell}{1+\ell m_+} &1\\1&0\end{pmatrix}.$$ 
As $\Phi^{-1}{}^*g$ is also projectively equivalent to $g$, by Proposition \ref{prop_proj_bande}  there exist $a$ and $c$ such that $
\Phi^{-1}{}^*g$ reads
$$
a\big(\frac{c}{(1+c f(x))^2} dx^2 +\frac{2}{1+c f(x)}dxdy +\frac{f(x)}{1+c f(x)}dy^2\big).
$$
Clearly $a=\frac{1}{1+\ell m_+}$ and  $c=\frac{-\ell}{1+\ell m_+}$.
Therefore, possibly replacing $\Phi$ by $\Phi^{-1}$, we can assume that $\ell<0$ (as $\ell>0$ implies $1+\ell m_+>0$).

We now compare the values of the integrals of the forms  $i_K \text {vol}_g$ and $i_K \text{vol}_{\Phi^*g}$ on the space of leaves of $K$. The transformation $\Phi$  preserving $K$, these values must be the same. If we denote by $dx$ the form $i_K \text {vol}_g$ then the form $i_K \text{vol}_{\Phi^*g}$ is   $\frac {1+\ell m_+}{(1+\ell f(x))^{\frac 32}} dx$, where $f$ is the expression of $g(K,K)$ in the coordinate $x$.

As $\ell<0$ and $f\geq 0$, for any $x$ we have 
$$0<(1+\ell f(x))^{3/2}<1+\ell f(x)<1+\ell m_+<1,$$
therefore for any $x$, we have $\frac {1+\ell m_+}{(1+\ell f(x))^{\frac 32}}>1$ and the two integrals are therefore different. We have a contradiction, therefore $\T$ has no Killing field.
\hop

If $(\T,g)$ has no Killing field then by \cite[Theorem 2]{Matveev_au_japon}, it is 
of Liouville type, therefore by  \cite[Lemma 1]{Matveev_au_japon} it is 
projectively equivalent to a Riemannian torus $(\T,g_0)$ also of Liouville type.
We recall that a  metric is of Liouville type if there exists two real periodic functions $h_1$ and $h_2$ and a lattice $\Lambda$ such that it is isometric to the quotient of $\R^2$ endowed with the metric  $(h_1(x)+h_2(y))(dx^2 \pm dy^2)$ by $\Lambda$. In the Riemannian case, the function $I_0\,:\, T\T\rightarrow \R$ defined by $I_0(v)=(h_1(x)+h_2(y))(h_1(y)dx^2-h_2(x)dy^2)$ is an integral of the geodesic flow of $g_0$, cf. \cite{Kiki} for example (the expression in \cite{Matveev_au_japon} looks different because the integral given there is on the cotangent bundle).

As $g_0$ is not flat, by  \cite[Corollary 1]{Matveev_au_japon}, the set of quadratic integrals of the geodesic flow of $g_0$ is $2$-dimensional and therefore spanned by $I_0$ and the energy. It follows by Theorem \ref{theo_classique} that the set $\mathcal P_0$  of Riemannian metrics projectively equivalent to $g_0$ is diffeomorphic to $\{(a,b)\in \R^2\,|\  ah_2(y)+b>0\ \text{and} -ah_1(x)+b>0,\ \forall (x,y)\in \R^2\}$ and therefore to  $\R^2$. Clearly, the identity component of the projective group of $g_0$ acts trivially on this set.

According to \cite{Zeghib}, the group of isometry of $g_0$ has finite index in the group of projective transformations. It implies that the group of projective transformations has a finite number of connected components. Hence the action of the projective group on $\mathcal P_0$ reduces to the action of a finite group. As this action is smooth it preserves a Riemannian metric. By the uniformization theorem, this action is conjugated to an action by biholomorphisms on $\C$ or the disk. Consequently, the action of the projective group fixes a metric of $\mathcal P_0$, we can assume that it is $g_0$. The set of non isometric projective transformations of $g$ is contained in the set of isometries of $g_0$ that do not preserve $I_0$. Theorem \ref{theo_fini} therefore follows from Lemma \ref{lem_isom_liouv}. 
\begin{lem}\label{lem_isom_liouv}
Let $g_0=(h_1(x)+h_2(y))(dx^2+dy^2)$ be a Riemannian metric of Liouville type on the torus.
The group of isometries of $g_0$ preserving the integral $I_0$ as an index lower or equal to $2$ in the isometry group of $g_0$. More precisely, there exists an   isometry  $\Phi$ of $g_0$ that does not preserve the Liouville integral $I_0$ if and only if there exist $k\in \R$ and $c\in \R$  such that $h_1$ is $2k$-periodic and $h_2(x+k)=h_1(x)+c$ or $h_2(-x-k)=h_1(x)+c$. Moreover in that case (after a possible change of origin) $\Phi(x,y)=(y+k,x+k)$ or $\Phi(x,y)=(-y+k,-x-k)$.
\end{lem}
\begin{proof}
An isometry $\Phi$  of $g_0$ is a conformal transformation of the flat metric induced by $dx^2+dy^2$ and therefore an isometry of it. It implies that the function $h_1+h_2$ is $\Phi$ invariant.
As $h_1+h_2$ is not constant, $\Phi$ preserves or permutes the lines $\R\partial_x$ and $\R\partial_y$, therefore  $\Phi_0$, the linear part of $\Phi$ in the coordinates $(x,y)$, is one of the following:
$$\pm \text{Id}, \pm \begin{pmatrix} 0&1\\ 1&0 \end{pmatrix}, \pm \begin{pmatrix} 0&1\\ -1&0 \end{pmatrix}, \pm \begin{pmatrix} 1&0\\ 0&-1 \end{pmatrix}.$$
If $\Phi_0\neq \pm \begin{pmatrix} 0&1\\ 1&0 \end{pmatrix}$ then, by a short computation, $\Phi$ preserves $h_1$ and $h_2$ (and not just their sum) and therefore $I_0$. It implies that the product of two isometries always preserves $I_0$ proving the first part of the lemma.

If $\Phi_0= \begin{pmatrix} 0&1\\ 1&0 \end{pmatrix}$ then (after a possible change of origin) there exists $k\in \R$ such $\Phi(x,y)=(y+k,x+k)$.
We therefore have $h_1(x)+h_2(y)=h_1(y+k)+h_2(x+k)$, for any $x$ and $y$. It follows that $h_1(x)-h_2(x+k)$ is constant and $h_1(x+2k)=h_1(x)$. Reciprocally, if $h_1(x)-h_2(x+k)$ is constant and $h_1(x+2k)=h_1(x)$ then the map $(x,y)\mapsto (y+k,x+k)$ is an isometry that does not preserve the Liouville integral.

If $\Phi_0= -\begin{pmatrix} 0&1\\ 1&0 \end{pmatrix}$ then (after a possible change of origin) there exists $k\in \R$ such $\Phi(x,y)=(-y+k,-x-k)$. 
We therefore have $h_1(x)+h_2(y)=h_1(-y+k)+h_2(-x-k)$, for any $x$ and $y$. It follows that $h_1(x)-h_2(-x-k)$ is constant and $h_1(x+2k)=h_1(x)$. Reciprocally, if $h_1(x)-h_2(-x-k)$ is constant and $h_1(x+2k)=h_1(x)$ then the map $(x,y)\mapsto (-y+k,-x-k)$ is an isometry that does not preserve the Liouville integral.
\end{proof}

\bigskip
\begin{tabular}{ll}
 Address: & Univ. Bordeaux, IMB, UMR 5251, F-33400 Talence, France\\
& CNRS, IMB, UMR 5251, F-33400 Talence, France\\
\\
E-mail:&{\tt pierre.mounoud@math.u-bordeaux.fr}
\end{tabular}

\end{document}